\documentclass[11pt]{article}       
\usepackage{graphicx}
\usepackage{amsthm}
\usepackage{latexsym}
\usepackage{amssymb}
\usepackage{amsmath}
\usepackage{amscd}

\newtheorem{definition}{Definition}
\newtheorem{lem}{Lemma}
\newtheorem{thm}{Theorem}
\newtheorem{corollary}{Corollary}
\newtheorem{prop}{Proposition}
\newtheorem{remark}{Remark}

\newcommand{\N}{\mathbb{N}}
\newcommand{\R}{\mathbb{R}}

\newcommand{\Z}{\mathbb{Z}}

\newcommand{\supp}{\mbox{supp }}

\begin{document}

\title{On Caffarelli-Kohn-Nirenberg inequalities for block-radial functions\thanks{The research of L. Skrzypczak was  supported in part by Wenner-Gren Foundations and by the  grant  UMO-2014/15/B/ST1/00164 of the National Center of Science, Poland.}
}

\author{Leszek Skrzypczak\\
	Faculty of Mathematics and Computer Science\\
	Adam Mickiewicz University,\\ 
	ul. Umultowska 87,\\
	61-879 Pozna\'n, Poland\\
	and\\
	Cyril Tintarev\\Uppsala University\\
	P.O.Box 480\\
	SE-751 06 Uppsala, Sweden}

\date{}

\maketitle

\begin{abstract}
The paper provides weighted Sobolev inequalities of the Caffa\-re\-lli-Kohn-Nirenberg type for functions with multi-radial symmetry. An elementary example of such inequality is the following inequality of Hardy type for functions $u=u(r_1(x),r_2(x))$, where $r_1(x)=\sqrt{x_1^2+x_2^2}$ and $r_2(x)=\sqrt{x_3^2+x_4^2}$ from the subspace $\dot{H}^{1,2}_{(2,2)}(\R^4)$ 
of the Sobolev space $\dot{H}^{1,2}(\R^4)$, radially symmetric in variables $(x_1,x_2)$ and in variables $(x_3,x_4)$:
$$
\int_{\R^4}\frac{u^2}{r_1(x)r_2(x)}dx\le C\int_{\R^4}|\nabla u|^2 dx,
$$
Similarly to the previously studied radial case, the range of parameters in CKN inequalities can be extended, sometimes to infinity, providing a pointwise estimate similar to the radial estimate in \cite{Strauss}. Furthermore, the ''multi-radial`` weights are a stronger singularity than radial weights of the same homogeneity, e.g. $ \frac{1}{r_1(x)r_2(x)}\ge \frac{1}{2|x|^2}$. 
\end{abstract}



\section{\label{intro}Introduction.}

\subsection{Subject of the paper}

Let $N>1$. A function $f$ on $\R^{N}$ is called block-radial, or multi-radial, or $m$-radial, $m>1$,
if $\R^{N}$ is considered as a product $\R^{\gamma_{1}}\times\dots\times\R^{\gamma_{m}}$,
$m\in\N$, and $f=f(r_{1},\dots,r_{m})$, where 
$$r_{j}=r_{j}(x)=\left(x_{\gamma_{1}+\dots+\gamma_{j-1}+1}^{2}+\dots+x_{\gamma_{1}+\dots+\gamma_{j-1}+\gamma_{j}}^{2}\right)^{1/2}.$$
In other words, we divide $\R^{N}$ into $m$ ``blocks'' of variables
$\R^{\gamma_{j}}$ and consider functions dependent on radii of those
blocks. If $m=1$, the function $f$ is called radial, and the condition
of block-radiality trivializes 
if $m=N$. For mnemonic reasons we will
often replace the notation $N$ with $|\gamma|:=\gamma_{1}+\dots+\gamma_{m}$. 

This paper studies embeddings of block-radial subspaces of Sobolev
spaces $\dot{H}^{1,q}$ into weighted $L^{q}$-spaces. As it transpires from Corollary~\ref{thm:4} below, an appropriate weight that estimates singularities  of block-radial functions near the hyperplanes $r_j(x)=0$, $j=1,\ldots ,m$, as well as their decay at infinity, has the form of a suitable power of the following function of homogeneity $1$:
\begin{equation}
	r_\gamma(x)\, =\,  r_{1}(x)^{\frac{\gamma_{1}-1}{|\gamma| - m}} r_{2}(x)^{\frac{\gamma_{2}-1}{|\gamma| - m}} \dots r_{m}(x)^{\frac{\gamma_{m}-1}{|\gamma| - m}}\, , \qquad x\in \R^N . \label{weight}
\end{equation}
Inequalities that express such embeddings are a part of the large family of Caffarelli-Kohn-Nirenberg
(CKN) inequalities \cite{CKN}, which includes the sharp Sobolev (Mazya-Talenti-Aubin)
inequality(for $q<N$) \cite{Aubin,Mazya-book,Talenti} and the $N$-dimensional 
Hardy inequality 
\begin{equation}
\int_{\R^{N}}\frac{|u(x)|^{q}}{|x|^{q}}dx\le\left(\frac{q}{|N-q|}\right)^{q}\int_{\R^{N}}|\nabla u|^{q}dx,\; u\in C_{0}^{\infty}(\R^{N}\setminus\{0\}) 
\label{eq:Hardy}
\end{equation}
$1\le q<\infty$ and $q\not= N$, cf. \cite[page 41]{Mazya-book}, where the vanishing of the functions at zero is in fact not required
for $q<N$. More generally if $N=m+n$, $m,n\in \N$, and $x=(y,z)$, $y\in \R^n$, $z\in\R^m$ then  
\begin{equation}
\int_{\R^{N}}\frac{|u(x)|^{q}}{|z|^{q}}dx\le\left(\frac{q}{|m-q|}\right)^{q}\int_{\R^{N}}|\nabla u(x)|^{q}dx,\; u\in C_{0}^{\infty}(\R^{N})
\label{eq:Hardy2}
\end{equation}
if $q\not= m$ and $u(y,0)=0$ for any $y\in \R^n$, or if $q>m$, cf. \cite[page 42]{Mazya-book}.

Interpolation between Hardy and Sobolev inequalities gives
immediately embeddings into weighted $L^{p}$-spaces with $q<p<q^{*}=\frac{qN}{N-q}$.
For the radial functions, however, the interpolated Hardy-Sobolev
inequalities also allow an extrapolation, with parameter $p$ taking
all values between $q$ and $\infty$ due to the Strauss-type  estimate
\begin{equation}
\text{\ensuremath{\sup}}_{x\in\R^{N}}|x|^{N-q}|u(|x|)|^{q}\le C\int_{\R^{N}}|\nabla u|^{q}dx,\label{eq:Strauss}
\end{equation}
$ 1\le q < N$, cf. \cite{Strauss,Lio82}. For more general inequalities for radial functions  we refer to, \cite{CO,SS,SSV,DeNapoli}.

In the $m$-radial case we deal with functions of $m$ variables,
and if $m\ge q$, there are no similar weighted $L^{\infty}$- bounds.
Instead, one expects embeddings of the $m$-radial subspaces of $\dot{H}^{1,q}$
into suitably weighted $L^{\frac{qm}{m-q}}$-spaces, with the ``supercritical''
exponent $q_{m}^{*}:=\frac{qm}{m-q}>q^{*}$. As we show, it is possible
to reduce this question to a ``block-radial'' analogue of the Hardy
inequality, which in many cases follows from the standard
Hardy inequality with help of the  H\"older inequality, giving the inequality \eqref{eq:newHardy1-3} below with the constant \eqref{easyC}. 
In addition to that we show that \eqref{eq:newHardy1-3} holds with a positive constant, not obtainable via H\"older inequality, even when the constant 
\eqref{easyC} is zero, provided that $q\ge 2$ and $1< m<|\gamma|$. In particular we have  
\[
\int_{\R^{2+1}}\frac{u^{2}}{r_{1}^2}dx\le C\int_{\R^{2+1}}|\nabla u|^{2}dx, 
\]
and 
\[
\int_{\R^{2+2}}\frac{u^{2}}{r_{1}r_{2}}dx\le C\int_{\R^{2+2}}|\nabla u|^{2}dx. 
\]

Another type of inequalities of CKN-type is the $m$-radial analogue
of (\ref{eq:Strauss}), expectable once $m<q$, since $m$-radial
functions in Sobolev spaces are continuous, with the growth or decay rate expressed by a power of $r_\gamma$, 
rather than a coarser estimate in terms of $|x|$.

Multiradial inequalities proved in this paper are scale-invariant, and thus the degree of homogeneity in the weight is uniquely determined by the exponent of the function or its gradient under the integral. In bounded domains this is not necessarily true, and we may mention in this connection an earlier result, Theorem~1.4 in \cite{fig}, which gives a similar inequality with a radial weight and biradial symmetry on a ball. This inequality is provided for an open interval of parameters, up to the endpoint value.

The objective of this paper, in addition to proving embeddings for $m$-radial subspaces
of Sobolev spaces, is to study compactness of such embeddings. A related
result for radial functions is compactness of embedding of the radial
subspace of $H^{1,q}$ into $L^{p}$, $q<p<q^{*}=\frac{qN}{N-q}$.
Here we prove compactness of embeddings of $m$-radial Sobolev spaces
into $L^{p}$ including the ``supercritical'' interval $p\in(q,q_{m}^{*})$,
 and discuss the structured loss of compactness at the ``$m$-critical
exponent'' $q_{m}^{*}=\frac{qm}{m-q}$ expressed in terms of a profile
decomposition.

Already from the argument in \cite{CKN} one can infer that the range of parameters in the CKN inequalities becomes greater if one restricts the class of functions to radial-symmetric ones. Radial CKN inequalities with the full range of parameters can be found in \cite{DeNapoli}. For this reason we restrict the consideration throughout the paper to $m$-radial functions with $2\le m<N$.  In \cite{Ding} Ding has described sufficient
conditions for compactness of embeddings for block-radial subspaces of Sobolev spaces similar to Strauss result on compactness. Block-radial symmetry is a natural conditions in the variational study of elliptic PDE, as a means to obtain non-radial solutions, while at the same time enjoying benefits of regularity that arises in problems in lower dimensions, allowing, in particular, higher rate of growth for nonlinear terms. We will not quote here numerous literature studying H\'enon's equation, which describes an model in astrophysics and has a nonlinearity with supercritical growth. Breaking of the radial symmetry of solutions in variational problems can be verified in some cases by comparing energy levels of functionals with radial versus $m$-radial symmetry. As examples of studies of problems with block-radial symmetry one may give \cite{AT1,AT2,Ku,PK}.

\subsection{Block-radial subspaces - definitions and notations.}

Let $m\in{1,\ldots ,N}$ and let $\gamma\in\N^{m}$ be  an $m$-tuple  $\gamma=(\gamma_1,\ldots, \gamma_m)$,  $\gamma_1+\ldots + \gamma_m = |\gamma|= N$. The $m$-tuple  $\gamma$ describes decomposition of $\R^{|\gamma|}= \R^{\gamma_{1}}\times\dots\times\R^{\gamma_{m}}$  into $m$ subspaces of dimensions $\gamma_{1},\dots,\gamma_{m}$ respectively.  Let 
$$SO(\gamma)= SO(\gamma_1)\times \ldots SO(\gamma_m)\subset SO(N)$$ 
be a group of isometries on $\R^{|\gamma|}$. An element $g=(g_1,\ldots, g_m)$, $g_i\in SO(\gamma_i)$  acts  on $x=(x_1,\ldots, x_m)$, $x_i\in \R^{\gamma_i}$ by $x\mapsto g(x)= (g_1(x_1),\ldots, g_m(x_m))$.  If $m=1$ then $SO(\gamma)= SO(N)$ is a special orthogonal group acting on $\R^N$. If $m=N$ then the group  is trivial  since then $\gamma_1= \ldots =\gamma_m = 1$ and $SO(1)= \{\mathrm{id}\}$.  

We will denote a subspace of any space $X$ of functions on $\R^{|\gamma|}$
consisting of functions invariant with respect to the action the group $SO(\gamma)$ as $X_{\gamma}$. If $SO(\gamma)=SO(N)$, then we will write $X_\mathrm{rad}$ since in that case the subspace consists of radial functions.  

The spaces of our concern here are homogeneous Sobolev spaces of invariant functions $\dot{H}_\gamma^{1,p}(\R^{|\gamma|})$ defined as the completion of 
$C_{0,\gamma}^{\infty}(\R^{|\gamma|})$ in the gradient norm $\|\nabla u\|_{p}$.  The space $\dot{H}_\gamma^{1,p}(\R^{|\gamma|})$ can be identified with the subspaces of  homogeneous Sobolev spaces  $\dot{H}^{1,p}(\R^{|\gamma|})$ defined as the completion of $C_0^\infty(\R^{|\gamma|})$.

Let $Y_k$ be a hyperplane in $\R^N$ of codimension $\gamma_k$ defined by $r_k=0$, and let us introduce the subset 
$$Y(\gamma)= \bigcup_{k: \gamma_k\ge 2}Y_k \subset \R^{|\gamma|}.$$ 

\begin{definition}
The homogeneous Sobolev space $\dot{H}_{0,\gamma}^{1,q}(\R^{|\gamma|})$ is the completion of $C_{0,\gamma}^{\infty}(\R^{|\gamma|}\setminus Y(\gamma))$ 
in the gradient norm $\|\nabla u\|_{q}$. 
\end{definition}
In fact, if $1<q\le\min\{\gamma_k:\gamma_k\ge 2, k=1,\dots,m\}$ and $1<m<|\gamma|$, then $\dot{H}_{0,\gamma}^{1,q}(\R^{|\gamma|})=\dot{H}_{\gamma}^{1,q}(\R^{|\gamma|})$, see for details the proof of Theorem \ref{thm:1}. 

In the case under study, $1<m<N$, as we show below (Lemma \ref{lem:4} and Lemma \ref{lem:5}), the subspaces of $\dot{H}_{0,\gamma}^{1,q}(\R^{|\gamma|})$
are still  spaces of functions. The larger spaces
 $\dot{H}_{\gamma}^{1,q}(\R^{|\gamma|})$ remain spaces of functions  whenever $1\le q < N$ (see e.g.\cite{Mazya-book}, p. 696).


\subsection{Organization of the paper and main results }

In Section 2 we prove preliminary estimates, including Lemma \ref{lem:2} with
a CKN-type inequality not contained in the main results of Section
3. 
In Section 3 we prove the following inequality of Hardy type.
\begin{thm} \label{thm:2}
Let $1<m<N$ and $1\le q<\infty$. 
Then there exists a constant $C>0$ such that for all 
 $u\in C_{0\gamma}^{\infty}(\R^{|\gamma|}\setminus Y(\gamma))$ the following inequality holds, 
\begin{equation}
\int_{\R^{|\gamma|}} \frac{|u(x)|^{q}}{r_\gamma(x)^{q}} \le C  \int_{\R^{|\gamma|}}|\nabla u(x)|^{q}dx\, .
\label{eq:newHardy1-3}
\end{equation}
Moreover 
\begin{equation}
\label{easyC}
 C\le 
 \begin{cases}
 \frac{q^q}{\prod_{i:\gamma_i\not=1}|q-\gamma_i|^{\frac{q(\gamma_i -1)}{|\gamma|-m}}}, &\qquad \text{if} \quad 1 \le q < 2 , \\
 C(\gamma)^q , &\qquad \text{if} \quad 2\le q < \infty .  
\end{cases} 
\end{equation}
\end{thm}
\noindent
We recall that the weight $r_\gamma(x)$ is defined by \eqref{weight}. 

For a subset of parameters one can extend the inequality (\ref{eq:newHardy1-3})
to the space $\dot{H}^{1,q}(\R^{|\gamma|})$:
\begin{thm} \label{thm:1} 
 Assume that $1<m<N$ 
 and $1< q \le \min\{\gamma_k: \gamma_k\ge 2\}$. 
Then there exists a positive constant $C$ such that for each
$u\in\dot{H}^{1,q}(\R^{|\gamma|})$, 
\begin{equation}
\int_{\R^{|\gamma|}}\frac{|u(x)|^{q}}{r_{\gamma}(x)^q}
dx\le C \int_{\R^{|\gamma|}}|\nabla u(x)|^{q}dx.
\label{eq:newHardy1-3-1}
\end{equation}
Moreover the constant $C$ satisfies the estimates \eqref{easyC}. 
\end{thm}

We prove also the following inequality of CKN type, which
extends the range of Sobolev embeddings beyond the critical exponent
$q^{*}=\frac{q|\gamma|}{|\gamma|-q}$:
\begin{thm} \label{thm:CKN} 
Let $1<m<N$, $1\le q<\infty$, $q\le p<\infty$,
and let $p\le q_{m}^{*}:=\frac{qm}{m-q}$ whenever $q<m$. 
Then there exists a
constant $C>0$, uniform with respect to $q\ge 2$,  such that for every $u\in\dot{H}_{0,\gamma}^{1,q}(\R^{|\gamma|})$,
\begin{equation}
\left(\int_{\R^{|\gamma|}}\left(\frac{|u(x)|}{r_\gamma(x)^{|\gamma|(\frac{1}{p}-\frac{1}{q})+1}}\right)^p
dx\right)^{q/p}\le C\int_{\R^{|\gamma|}}|\nabla u(x)|^{q}dx.\label{eq:CKN}
\end{equation}
If $q\not= m $ then the constant $C>0$ is independent of $p$. 

If $q <  \min\{\gamma_k: \gamma_k\ge 2\}$ then the inequality \eqref{eq:CKN} holds for any $u\in\dot{H}_\gamma^{1,q}(\R^{|\gamma|})$.  Moreover if $q <  \min\{\gamma_k: \gamma_k\ge 2\}$ and $p<q^*$ 
then the inequality \eqref{eq:CKN} holds for any $u\in\dot{H}^{1,q}(\R^{|\gamma|})$. 
\end{thm}

By taking $p\to\infty$ in (\ref{eq:CKN}), we have
immediately the following $m$-radial analogue of the Strauss estimate for radial functions (which corresponds formally to the case $m=1$ in the inequality below). The statement is repeated in the second part of corollary for $\dot{H}_\gamma^{1,q}(\R^{|\gamma|})$ when it coincides with $\dot{H}_{0,\gamma}^{1,q}(\R^{|\gamma|})$. 

\begin{corollary}\label{thm:4}
Let $1<m<q<|\gamma|$. 

\noindent
(i) There exists a positive constant $C(\gamma)$ such that for any $u\in\dot{H}_{0,\gamma}^{1,q}(\R^{|\gamma|})$,
\begin{equation} 
\sup_x (r_{\gamma}(x))^{|\gamma|-q}|u(x)|^q \le C(\gamma)^{q}\int_{\R^{|\gamma|}}|\nabla u(x)|^q dx.\label{eq:newHardy1-3-2}
\end{equation}
Moreover $u$ is a continuous function outside the set $Y_\gamma$. 

\noindent
(ii) If  $m< q <  \min\{\gamma_k: \gamma_k\ge 2\}$ then the inequality \eqref{eq:newHardy1-3-2} holds for any $u\in\dot{H}_\gamma^{1,q}(\R^{|\gamma|})$ and $u$ is a continuous function outside the set $Y_\gamma$.  
\end{corollary}

Somewhat unexpectedly, inequality (\ref{eq:newHardy1-3-2}) does not generally hold for functions in $\dot{H}_\gamma^{1,q}(\R^{|\gamma|})$, despite the fact that when $N>q>m$, elements of $\dot{H}_\gamma^{1,q}(\R^{|\gamma|})$ are continuous away from $Y(\gamma)$. The meaning of the following counterexample, proved at the end of Section 3 is that the functions in $\dot{H}_\gamma^{1,q}(\R^{|\gamma|})$ may have asymptotics near $Y(\gamma)$ incomparable to that for functions belonging to $\dot{H}_{0\gamma}^{1,q}(\R^{|\gamma|})$.

\begin{prop}
\label{counterex} Inequality (\ref{eq:newHardy1-3-2}) does not hold for the spaces  $\dot{H}_\gamma^{1,3}(\R^{4})$ with $\gamma_1=\gamma_2=2$.
\end{prop}
 It is easy to adapt the argument of the counterexample to some other values of $\gamma$ and $q$. 

\section{Preliminary estimates}

We start with the following lemma. 
\begin{lem}
\label{lem:newHardy}
Let $m=2$, $\max\{\gamma_1,\gamma_2\}\ge 2$ 
and $\alpha\in\R$.
There exists a constant $C_{1}>0$ 
such that for all $u\in C_{0\gamma}^{\infty}(\R^{|\gamma|}\setminus Y(\gamma))$, 
\begin{equation}
\int_{\R^{|\gamma|}}\frac{|\nabla u(x)|^{2}}{(r_{1}(x)^{\gamma_1-1}r_{2}(x)^{\gamma_2-1})^\frac{\alpha}{|\gamma|-2}}dx\ge C_{1} 
\int_{\R^{|\gamma|}} \frac{|u(x){}|^{2}}{(r_{1}(x)^{\gamma_1-1}r_{2}(x)^{\gamma_2-1})^\frac{\alpha+2}{|\gamma|-2}} dx. 
\label{eq:newHardy1}
\end{equation}
Moreover there exists a positive constant $C_\gamma$ depending on $\gamma$, such that  $C_{1}\ge C_{\gamma}\alpha^{2}$ for all $\alpha$ sufficiently large.  
\end{lem}

\begin{proof}
The function $u$ is block-radial, therefore the inequality \eqref{eq:newHardy1} is equivalent to 
\begin{equation} \nonumber
\int_0^\infty\hspace{-2mm}\int_0^\infty\frac{|\nabla u(r_1,r_2)|^2}{(r_{1}^{\gamma_1-1}r_{2}^{\gamma_2-1})^{\frac{\alpha}{|\gamma|-2}-1}} dr_1 dr_2 \ge C_1
\int_0^\infty\hspace{-2mm}\int_0^\infty\frac{|u(r_1,r_2)|^2}{(r_{1}^{\gamma_1-1}r_{2}^{\gamma_2-1})^{\frac{\alpha+2}{|\gamma|-2}-1}} dr_1 dr_2. 
\end{equation}
We assume that $ \max\{\gamma_1,\gamma_2\}=\gamma_{2}$ so that  $\gamma_{2}\ge 2$. The case $\max\{\gamma_1,\gamma_2\}=\gamma_{1}$ then follows by renumbering the variables.  
Let us provide the quadrant $(0,\infty)^2$ with polar coordinates by setting $r_1=r\cos\theta$, $r_2=r\sin\theta$, $0\le\theta\le\pi/2$. In these coordinates we have 
$$
|\nabla_{r_1,r_2} u(r_1,r_2)|^2=|\partial_ru(r,\theta)|^2+\frac{|\partial_\theta u(r,\theta)|^2}{r^2}\ge \frac{|\partial_\theta u(r,\theta)|^2}{r^2}.
$$ 
It suffices therefore to prove the following inequality under the integral with respect to $r$, using the notation $\psi(\theta)=(\cos\theta)^{\gamma_1-1}(\sin\theta)^{\gamma_2-1}$:
\begin{equation}
\int_0^{\pi/2}
|\partial_\theta u(r,\theta)|^2\psi(\theta)^{1-\frac{\alpha}{|\gamma|-2}} d\theta \ge C_1
\int_0^{\pi/2}|u(r,\theta)|^2\psi(\theta)^{1-\frac{\alpha+2}{|\gamma|-2}} d\theta , 
\label{eq:Hardy-theta}
\end{equation}
$r>0$, with the boundary condition $u(r,0)=u(r,\pi/2)=0$ or $u(r,0)=0 $ if $\gamma_1=1$.
We will show that \eqref{eq:Hardy-theta} follows from the Hardy inequality in one dimension. 

First we assume that $\gamma_1>1$. An elementary calculation shows that $\theta_\gamma =\arctan\sqrt{\frac{\gamma_2-1}{\gamma_1-1}}$ is a point of internal maximum for function $\psi(\theta)$, with the negative second derivative, and that $\psi(\theta)$ is increasing on the interval $(0,\theta_\gamma)$ and is decreasing on $(\theta_\gamma,\pi/2)$.   
Consider now the inequalities
\begin{equation}
\int_0^{\theta_\gamma}
|u_\theta(r,\theta)|^2\psi(\theta)^{1-\frac{\alpha}{|\gamma|-2}} d\theta \ge C(\beta)
\int_0^{\theta_\gamma}|u(r,\theta)|^2\psi(\theta)^{1-\frac{\alpha+2}{|\gamma|-2}} d\theta, \;r>0,
\label{eq:Hardy-theta-gamma}
\end{equation}
and 
\begin{equation}
\int_{\theta_\gamma}^{\pi/2}
|u_\theta(r,\theta)|^2\psi(\theta)^{1-\frac{\alpha}{|\gamma|-2}} d\theta \ge C(\beta)
\int_{\theta_\gamma}^{\pi/2}|u(r,\theta)|^2\psi(\theta)^{1-\frac{\alpha+2}{|\gamma|-2}} d\theta, \;r>0.
\label{eq:Hardy-theta-gamma1}
\end{equation}
Once we prove the inequalities, with respective conditions $u(r,0)=0$,  and $u(r,\pi/2)=0$, we have \eqref{eq:Hardy-theta} which yields the assertion of the lemma.   
We prove below only \eqref{eq:Hardy-theta-gamma}, since  \eqref{eq:Hardy-theta-gamma1} follows from \eqref{eq:Hardy-theta-gamma} by interchanging role of the variables $r_1$ and $r_2$ i.e. taking $r_1=r\sin\theta$, $r_2=r\cos\theta$.

Let $\nu=\left(\frac{\alpha}{|\gamma|-2}-1\right)(\gamma_2-1)$. We introduce a mapping $t(\theta)$ as a solution of equation
\begin{equation}
t'(\theta) = {\mathrm{ sign}}(\nu+1) \psi(\theta)^{\frac{\alpha}{|\gamma|-2}-1}, 
\end{equation}
where ${\mathrm{sign}}(x) = 1$ if $x>0$ and $-1$ otherwise.  We set the initial condition for the solution according to values of the parameters $\alpha$ and $\gamma$, by taking into account that $\psi(\theta)=\theta^{\gamma_2-1}(1+o_{\theta\to 0}(1))$. 
 Then $t'(\theta)=\theta^\nu( {\mathrm{ sign}}(\nu+1)+o_{\theta\to 0}(1))$. 
If $\nu>-1$, we set $t(0)=0$, so that $t$ is a monotone-increasing bijection between $(0,\theta_\gamma)$ and $I=(0,t_\gamma)$ with some $t_\gamma\in(0,\infty)$. If $\nu\le-1$, we set $t(0)=+\infty$, so that $t$ is a monotone-decreasing bijection between $(0,\theta_\gamma)$ and $I=(t_\gamma,\infty)$ with some $t_\gamma\in\R$.     
A substitution $\theta\mapsto t(\theta)$ into \eqref{eq:Hardy-theta-gamma}, gives us 
\begin{equation}\label{eq:Hardy-theta-gamma_a} 
\int_I |w'(t)|^2dt\ge C\, \int_I |w(t)|^2V(t)dt,\;w(0)=0,
\end{equation}
where the weight 
$$
V(t) = \psi(\theta(t))^{1-\frac{\alpha+2}{|\gamma|-2}}\theta'(t) =  {\mathrm{ sign}}(\nu+1) t'(\theta(t))^{-2} \psi(\theta(t))^{-\frac{2}{|\gamma|-2}}
$$ 
is singular either at zero, when $\nu>-1$, or at infinity, when $\nu\le -1$. It suffices therefore to verify that $V(t)=O(t^{-2})$ when $t$ goes respectively to zero or to infinity. In fact, we get $V(t)=o(t^{-2})$. Computation of the asymptotic behaviour of $V(t)$ and estimation of the coefficient $ C_1(\alpha)\ge C_\gamma\alpha^{2}$ for large $\alpha$ is straightforward and is left to the reader. 
Here we give only some elaboration for the case $\nu=-1$. In this case $t(\theta)=\log\frac{1}{\theta}(1+o(1))$,  $V(t)=e^{-(\gamma_2-1)\left(2-\frac{2\alpha+2}{|\gamma|-2}\right)t}(1+o(1))$,
and the coefficient in the exponent is negative. When $\nu\neq -1$, $V(t)$ has a two-sided estimate by a suitable power of $t$.

The inequality \eqref{eq:Hardy-theta-gamma_a} follows from one-dimensional Hardy inequality and the estimate $V(t)=O(t^{-2})$.

At the end let  $\gamma_1=1$. Then  $\theta_\gamma=\pi/2$, so it is sufficient to consider the integral from zero to $\frac{\pi}{2}$ with the boundary condition $u(r,0)=0$, cf. \eqref{eq:Hardy-theta-gamma}. We proceed in the similar way to the former case. \qed
\end{proof}

\begin{lem}\label{lem:2}
Let $m=2$,  $\max\{\gamma_1,\gamma_2\}\ge 2$, 
$q\in [2,\infty)$, and $\beta\in\R$.
There exists a constant $C_{2}>0$  
such that for all $u\in C_{0\gamma}^{\infty}(\R^{|\gamma|}\setminus Y(\gamma))$, 
\begin{equation}
\int_{\R^{|\gamma|}} \frac{|u(x)|^{q}}{(r_{1}(x)^{\gamma_1-1}r_{2}(x)^{\gamma_2-1})^\frac{\beta+q}{|\gamma|-2}} dx\le C_2 q^q
\int_{\R^{|\gamma|}}\frac{|\nabla u(x)|^{q}}{(r_{1}(x)^{\gamma_1-1}r_{2}(x)^{\gamma_2-1})^\frac{\beta}{|\gamma|-2}}dx.
\label{eq:newHardy1-1}
\end{equation}
Moreover,  $C_{2}\le \left(\frac{C(\gamma)}{\beta+q}\right)^{q}$ for all sufficiently large values of $\beta+q$.  
\end{lem}

\begin{proof}
Since the values of the left and the right hand side do not change
from the replacement of $u$ by $|u|$ (and since we can use approximation
of Lipschitz functions by smooth functions), assume without loss of
generality that $u\ge0$.  
Applying   Lemma \ref{lem:newHardy} to $v=u^{q/2}$, with $\alpha=\beta+q-2$
we have, with the same constant $C_1$, 
\begin{align}
\nonumber
\int_{\R^{|\gamma|}}\frac{u(x)^{q}}{r_\gamma(x)^{\beta+q}}dx\,  =\, \int_{\R^{|\gamma|}}\frac{v(x){}^{2}}{r_\gamma(x)^{\alpha+2}}dx &\le \\ 
  \qquad\qquad\le\,  C_{1}^{-1}\int_{\R^{|\gamma|}}\frac{|\nabla u(x)^{q/2}|^{2}}{r_\gamma(x)^{\alpha}}dx &=\, C_{1}^{-1}\frac{q^{2}}{4}\int_{\R^{|\gamma|}}\frac{|\nabla u(x)|^{2}u(x)^{q-2}}{r_\gamma(x)^{\beta+q-2}}dx.
\nonumber
\end{align}

Applying H\"older inequality with exponents $\frac{q}{2}$, $\frac{q}{q-2}$,
we get 
\begin{align}
\nonumber
 \int_{\R^{|\gamma|}}\frac{u(x)^{q}}{r_\gamma(x)^{\beta+q}}dx & \le \;
 C_{1}^{-1}\frac{q^{2}}{4}\int_{\R^{\gamma}}\frac{u(x)^{q-2}|\nabla u(x)|^{2}}{r_\gamma^{(\beta+q)\frac{q-2}{q}}r_\gamma^{2\frac{\beta}{q}}} dx \le
\\ 
&\le  C_{1}^{-1}\frac{q^{2}}{4}\left(\int_{\R^{|\gamma|}}\frac{|u(x)|^{q}}{r_\gamma(x)^{\beta+q}} dx \right)^{1-\frac{2}{q}}\left(\int_{\R^{|\gamma|}}\frac{|\nabla u(x)|^{q}}{r_\gamma(x)^{\beta}} dx \right)^{\frac{2}{q}},
\nonumber
\end{align}
from which (\ref{eq:newHardy1-1}) is immediate. \qed
\end{proof}
\begin{lem}
\label{lem:4} Let $1<m<|\gamma|$ and $q\in[2,\infty)$. There exists a constant $C_{3}>0$ 
such that for all $u\in C_{0\gamma}^{\infty}(\R^{|\gamma|}\setminus Y(\gamma))$, 
\begin{equation}
\int_{\R^{|\gamma|}} \frac{|u(x)|^{q}}{r_\gamma(x)^{q}}\le C_3q^q 
\int_{\R^{|\gamma|}}|\nabla u(x)|^{q}dx,
\label{eq:Hardy>2}
\end{equation}
where $r_\gamma(x)={(r_{1}(x)^{\gamma_1-1}\dots r_{m}(x)^{\gamma_m-1})^\frac{1}{|\gamma|-m}}$.

Moreover,  $C_{3}\le \left(\frac{C(\gamma)}{q}\right)^{q}$ for all sufficiently large values of $q$.  
\end{lem}

\begin{proof}
Without loss of generality assume that the dimensions  $\gamma_i$ are descending, and let $j\in\N$ be the largest value of $i$ such that $\gamma_i\ge 2$. Assume first that $j\in\{1,2\}$. Apply  \eqref{eq:newHardy1-1} with integration over $\R^{\gamma_1+\gamma_2}$, integrate with respect to the remaining variables and augment the gradient in the right hand side by derivatives with respect to the remaining variables.

Let now $j\ge 3$. We use the representation 
\[
r_1^{\gamma_1-1}\dots r_j^{\gamma_j-1}=
\left(r_j^{\gamma_j-1}r_1^{\gamma_1-1}\right)^{1/2} \left(r_1^{\gamma_1-1}r_2^{\gamma_2-1}\right)^{1/2}\dots \left(r_{j-1}^{\gamma_{j-1}-1}r_j^{\gamma_j-1}\right)^{1/2}
\] 
and applying the H\"older inequality.  Let $\gamma_0=\gamma_j$ and 
$$
p_k = 2\frac{|\gamma| - m }{\gamma_{k-1}+\gamma_k -2} = 2\frac{\gamma_1+\cdots + \gamma_j - j }{\gamma_{k-1}+\gamma_k -2}, \qquad k=1,\ldots , j\, . 
$$ 
Then 
$$
\sum_{k=1}^{j} \frac{1}{p_k} = 1\, .
$$
so we can use the H\"older inequality to the product to the product of $j$ terms
\[
\frac{|u|^\frac{q}{p_1}}{\left(r_j^{\gamma_j-1}r_1^{\gamma_1-1}\right)^\frac{q}{2(|\gamma|-m)}}\dots
\frac{|u|^\frac{q}{p_j}}{\left(r_{j-1}^{\gamma_j-1}r_j^{\gamma_j-1}\right)^\frac{q}{2(|\gamma|-m)}}.  
\]
Then for any factor we can used Lemma \ref{lem:2} with $\beta=0$ since 
$\frac{p_k q}{2(|\gamma|-m)}= \frac{q}{\gamma_{k-1}+\gamma_k-2}$ .  In consequence we get 
\begin{align}
\int_{\R^{\gamma_1+\ldots +\gamma_j}} \frac{|u(x)|^{q}}{r_\gamma(x)^{q}} &\le C_3\,  q^q 
\left(\int_{\R^{\gamma_j+\gamma_1}}|\nabla u(x)|^{q} \right)^{\frac{1}{p_1}} \ldots \left(\int_{\R^{\gamma_{j-1}+\gamma_j}}|\nabla u(x)|^{q} \right)^{\frac{1}{p_j}}  \nonumber\\ 
& \le C_3\,  q^q \int_{\R^{\gamma_1+\ldots +\gamma_j}} |\nabla u(x)|^{q}
\label{eq:Hardy>2a}
\end{align}

Now the inequality \eqref{eq:Hardy>2} follows easily from  \eqref{eq:Hardy>2a} by integration  with respect to the remaining variables and augmenting the gradient, if it is necessary.   \qed
\end{proof}
\begin{lem}
\label{lem:5} Let  $1<m<|\gamma|$ and $1 \le q < 2$. 
Then for all $u\in C_{0\gamma}^{\infty}(\R^{|\gamma|}\setminus Y(\gamma))$, 
\begin{equation}
\label{eq:Hardy<2}
\int_{\R^{|\gamma|}} \frac{|u(x)|^{q}}{r_\gamma(x)^{q}}\le \frac{q^q}{\prod_{i:\gamma_i\not= 1}|q-\gamma_i|^{\frac{q(\gamma_i-1)}{|\gamma|-2}}} 
\int_{\R^{|\gamma|}}|\nabla u(x)|^{q}dx,
\end{equation}

where $r_\gamma(x)={(r_{1}(x)^{\gamma_1-1}\dots r_{m}(x)^{\gamma_m-1})^\frac{1}{|\gamma|-m}}$
\end{lem}

\begin{proof}
Once more we assume that the dimensions  $\gamma_i$ are descending, and let $j\in\N$ be the largest value of $i$ such that $\gamma_i\ge 2$.
From Hardy inequalities for radial functions in $\R^{\gamma_{i}}$, $1\le i\le j$, we have: 
\begin{equation}
\int_{\R^{|\gamma|}}\frac{|u(x)|^{q}}{r_{i}(x)^{q}}dx\le
\left|\frac{q}{q-\gamma_{i}}\right|^{q}\int_{\R^{|\gamma|}}|\nabla_{i}u(x)|^{q} dx , 
\label{eq:lowH1}
\end{equation}
where $\nabla_i$ denotes the gradient with respect to the variables $x_i$ with $i={\sum_{j\le i-1}\gamma_j+1},\ldots , {\sum_{j\le i}\gamma_j}$. 
We represent the integrand  of the left hand side of (\ref{eq:Hardy<2}) as  a product of $m$ terms $\frac{|u(x)|^{\sigma_i q}}{r_{i}(x)^{q\sigma_i}}$,  with $\sigma_i=\frac{\gamma_i-1}{|\gamma|-m}$. Please note that only the indexes $i$ between $1$ and $j$ are relevant and that $|\gamma|-m = \gamma_1+\ldots \gamma_j- j$. So we can apply   the H\"older inequality with the exponent $1/\sigma_i$, $1\le i\le j$, and afterwards replace partial gradients $\nabla_i$ with with the full gradient in the right hand side. In this way  we immediately arrive at (\ref{eq:Hardy<2}). \qed
\end{proof}

\section{Proofs of the main results}
\subsection{Inequalities of Hardy-type}
The proof of Theorem \ref{thm:2} follows immediately from Lemma~\ref{lem:4} and Lemma~\ref{lem:5}.  

\begin{corollary}
\label{cor:gradient} Let $1<m<N$. There exists $C(\gamma)>0$, such that for all $u\in C_{0\gamma}^{\infty}(\R^{|\gamma|}\setminus Y(\gamma))$
\begin{equation}
\|u/r_\gamma\|_\infty \le C_\gamma \|\nabla u\|_\infty.
\label{eq:Hardy-infty}
\end{equation}
\end{corollary}
\begin{proof} Since the constant in \eqref{eq:newHardy1-3} is independent of $q$ we may pass to the limit as $q\to\infty$. The second statemt follows immediately from the first one and the definition of $r_\gamma$. \qed
\end{proof}

Note that the pointwise estimate of functions in $\dot{H}_{0,\gamma}^{1,\infty}(\R^{|\gamma|})$ given by Corollary~\ref{cor:gradient} is 
of the form $|u(r_{1},\dots,r_{m)}|\le Cr_{1}^{\alpha_{1}}\dots r_{m}^{\alpha_{m}}$
with $\alpha_{i}\in[0,1]$, $\sum\alpha_{i}=1$, which is sharper
than $|u(x)|\le C|x|$ for general Lipschitz functions. This is a consequence of having the
zero value on $\{r_{1}\dots r_{m}=0\}$.

{\em Proof of Theorem \ref{thm:1}.}

1. We prove first that under assumptions of the theorem  $C^\infty_{0,\gamma}(\R^N\setminus Y(\gamma))$ is dense in $\dot{H}_{\gamma}^{1,q}(\R^{|\gamma|})$. 
   
Let $H^{1,q}(\R^N)$ denote the usual inhomogeneous Sobolev space on $\R^N$ and $H_\gamma^{1,q}(\R^N)$ its subspace consisting of $SO(\gamma)$ invariant functions. It is known that  $H_\gamma^{1,q}(\R^N)$ is a complemented subspace of $H^{1,q}(\R^N)$, cf. \cite{LS}. Moreover,     $C^\infty_{0,\gamma}(\R^N)$ is dense in $H_\gamma^{1,q}(\R^N)$ and the dual space $(H_\gamma^{1,q}(\R^N))'$ can be identified with $H_\gamma^{-1,q'}(\R^N)$. 

Every Cauchy sequence in  $H^{1,q}(\R^N)$ is a Cauchy sequence in the sense of the gradient norm. Therefore  $H_\gamma^{1,q}(\R^N)$ can be embedded into   $\dot{H}_{\gamma}^{1,q}(\R^{|\gamma|})$. In consequence it is sufficient to show that $C^\infty_{0,\gamma}(\R^N\setminus Y(\gamma))$ is dense in 
$H_\gamma^{1,q}(\R^N)$. 
Let us assume  that the space $C^\infty_{0,\gamma}(\R^N\setminus Y(\gamma))$ is not  dense.  
By the Hahn-Banach theorem  there exists $f\in H_\gamma^{-1,q'}(\R^N)$, $f\not= 0$, such that $f(\varphi)=0$ for any $\varphi\in C^\infty_{0,\gamma}(\R^N\setminus Y(\gamma))$.  But this means that there exists a non-zero element $f$ of $H_\gamma^{-1,q'}(\R^N)$ with $\supp f \subset Y(\gamma)$. 
This implies that the inner capacity $\underline{\mathrm{Cap}}(Y(\gamma), H^{1,q})$ is strictly positive (see  Chapter 13.2 in  \cite{Mazya-book}). On the other hand $N-q\ge N-\gamma_k$ for any $\gamma_k\ge 2$. So the Hausdorff measure ${\mathcal H}_{N-q}$ of any compact subset of $Y(\gamma)$ is finite. In consequence   Proposition 10.4.3/3 and Theorem 13.3/2 in \cite{Mazya-book} imply $\underline{\mathrm{Cap}}(Y(\gamma), H^{1,q})=0$. This give us the contradiction.

2. Let now $u\in C^\infty_0(\R^{|\gamma|})$. We define   its   iterated spherical rearrangement $u^\star$  as follows. Let $u^\star_1(r_1, x_{\gamma_1+1},x_N)$ be the symmetric-decreasing rearrangement of $u(\cdot, x_{\gamma_1+1},\dots ,x_N)$ in $\R^\gamma_1$ with the values of $x_{\gamma_1+1},\ldots ,x_N$ fixed. Let us assume that 
$$
u^\star_{i-1}(r_1,\dots,r_{i-1}, x_{\sum_{j=1}^{i-1}\gamma_j+1},\dots ,x_N)
$$
is already described.  The function $u^\star_{i-1}$ admits a finite value at any point $(r_1,\dots,r_{i-1}, x_{\sum_{j=1}^{i-1}\gamma_j+1},\dots ,x_N)$. 
We define  
 $$
 u^\star_i(r_1,\dots,r_i, x_{\sum_{j=1}^i\gamma_j+1},\dots ,x_N)
 $$ 
to  be the symmetric-decreasing rearrangement of $u^\star_{i-1}$ in $\R^{\gamma_i}$ with the values of $r_1,\dots, r_{i-1}$ and  $x_{\sum_{j=1}^i\gamma_j+1},\ldots ,x_N$ fixed. Finally, set $u^\star = u_m^\star$. 

By the Fubini theorem and the Polya-Szeg\"o inequality applied  consecutively to $u^\star_i$, $i=m,m-1,\dots 1$ we get 
\begin{align}
	\label{PSit}
\int_{\R^N}|\nabla u^\star|^qdx & =  \int_{\ldots} \int_{\R^{\gamma_m}}|\nabla u_m^\star|^qdx \le 
\int_{\ldots} \int_{\R^{\gamma_m}}|\nabla u_{m-1}^\star|^q dx \\ 
& = \int_{\R^N}|\nabla u_{m-1}^\star|^qdx \le\; \ldots\;  \le \int_{\R^N}|\nabla u|^qdx, \nonumber 
\end{align}
which implies that  $u^\star\in \dot{H}_{\gamma}^{1,q}(\R^{|\gamma|})$. 

Now by the Hardy-Littlewood inequality, applied at the $i$th block we get  
\begin{align}
&\int_{\R^{\gamma_i}}\frac{|u_{i-1}(r_1,\dots,r_i-1,y, x_{\sum_j=1^i\gamma_j+1},x_N)|^q}{|y|^{q\frac{\gamma_i-1}{|\gamma|-m}}}dy \le \\
&\qquad\qquad\qquad\qquad\le \int_{\R^{\gamma_i}}\frac{|u_i^\star(r_1,\dots,r_i, x_{\sum_j=1^i\gamma_j+1},x_N)|^q}{|y|^{q\frac{\gamma_i-1}{|\gamma|-m}}}dy, \nonumber 
\end{align}
where $y$ denotes the variables of the $i$th block (i.e. $|y|=r_i$). Please note that the function $y\mapsto |y|^{q\frac{\gamma_i-1}{|\gamma|-m}}$ is radial and  decreasing. 
Once more using the Fubini theorem   we arrive at 
\begin{equation}
\label{HLit}
\int_{\R^N}  \frac{|u(x)|^q}{r_\gamma^q(x)}dx\le \int_{\R^N}  \frac{|u^\star|^q}{r_\gamma^q(x)}dx. 
\end{equation}
Combining \eqref{PSit} with \eqref{HLit} we get \eqref{eq:newHardy1-3-1}. 
\hfill $\Box$

\subsection{Inequalities of Caffarelli-Kohn-Nirenberg type.} We prove here Theorem \ref{thm:CKN} and Proposition \ref{counterex}.

{\em Proof of Theorem \ref{thm:CKN}.}
Let $u\in C^\infty_{0,\gamma}(\R^{|\gamma|})$. Changing  variables we get by the block-radiality of the function $u$ that $|\nabla_x u(x)| = |\nabla_{r_1,\ldots , r_m} u(r_1,\ldots, r_m)|$ and 
\begin{equation}\label{CKN01}
\int_{\R^{|\gamma|}}|\nabla u(x)|^{q}dx  = C \int_0^\infty\hspace{-2mm} \ldots \int_0^\infty |\nabla u(r_1,\ldots, r_m)|^q r_1^{\gamma_1-1}\hspace{-2mm}\ldots r_m^{\gamma_m-1} dr_1\ldots dr_m . 
\end{equation}
Let 
$\omega=\{(r_1,\ldots, r_m)\in\R^{m}:\, r_1^{\gamma_1-1}\cdot \ldots \cdot r_m^{\gamma_m-1}\in(1,2^{|\gamma|-m})\}$. 
Then $\omega$ is a domain in $\R^m$ with a uniformly Lipschitz boundary, and therefore it is an extension domain, \cite[Theorem 12.15]{Leo}. This means that the Sobolev space $H^{1,q}(\omega)$ defined by restrictions can be embedded into $L^p(\omega)$:
\begin{equation}\label{CKN02}
\left(\int_{\omega}|u|^{p}dr_1\ldots dr_m \right)^{\frac{q}{p}} \le C \left(\int_{\omega}|\nabla u|^{q} dr_1\ldots dr_m + \int_{\omega}|u|^{q}dr_1\ldots dr_m\right).	
\end{equation}
Moreover, if  $q\not= m $ then the constant is independent of $p$. For $q>m$ this follows from the embedding of $H^{1,q}(\omega)$ into $L^\infty(\omega)$, and  for $q<m$ it follows from the Talenti's results \cite{Talenti}, cf. also \cite[Corollary 11.9]{Leo}. In both cases the uniform constant is a consequence of the H\"older inequality and the embedding at the endpoint values $p=q$ and $p=q^*$ (understood as $p=\infty$ when $q>m$. Note that there is no endpoint embedding at $q=m$), and thus, necessarily, there is no uniform $L^p$-bound. 
Let $\Omega=\{x\in\R^{|\gamma|}:\, r_\gamma(x)\in(1,2)\}$. Then $x\in \Omega$ if and only if $(r_1(x),\ldots, r_m(x))\in \omega$. Now the definition of $\omega$, \eqref{CKN01}, and \eqref{CKN02} impliey 
\begin{equation}
	\left(\int_{\Omega}|u|^{p}dx\right)^{\frac{q}{p}}\le C\left(\int_{\Omega}|\nabla u|^{q}dx+\int_{\Omega}|u|^{q}dx\right) ,
\end{equation}
and the constant $C$ is independent of $p$ if $q\not=m$.

Rescaling $\Omega$ by the factor $2^{-j}$, $j\in\Z$, we have 
\[
C\left(2^{|\gamma|j}\int_{2^{-j}\Omega}|u|^{p}dx\right)^{\frac{q}{p}}\le2^{(|\gamma|-q)j}\int_{2^{-j}\Omega}|\nabla u|^{q}dx+2^{|\gamma|j}\int_{2^{-j}\Omega}|u|^{q}dx.
\]
Note that $1\le(2^j r_\gamma(x))^{|\gamma|-m}\le2$ whenever $x\in2^{-j}\Omega$.
So multiplying the above inequality  by $2^{(q-|\gamma|)j}$, replacing
the powers of $2$, taken under the integral, by appropriate powers
of $r_\gamma(x)$, and adding up the inequality over $j\in\Z$, we get 
\[
\left(\int_{\Omega}\frac{|u|^{p}}{r_\gamma(x)^{|\gamma|(1-p/q)+p}}dx\right)^{\frac{q}{p}} \le C\left(\int_{\R^{|\gamma|}}|\nabla u|^{q}dx+\int_{\R^{|\gamma|}}\frac{|u|^{q}}{r_\gamma(x)^{q}}dx\right).
\]
By Theorem \ref{thm:2}, the right hand side is bounded by $C\|\nabla u\|_{q}^{q}$. 

The argument extending the inequality to the space $\dot{H}_\gamma^{1,q}(\R^{|\gamma|})$ is the same as in the first  part of the proof of Theorem~\ref{thm:1}. If moreover $\frac 1q -  \frac {1}{|\gamma|}<\frac 1p $ one can also use the second part of the proof of Theorem~\ref{thm:1}, and extend the inequality to $\dot{H}_\gamma^{1,q}(\R^{|\gamma|})$. The details are omitted. 
\hfill $\Box$

{\em Proof of Proposition~\ref{counterex}.} It suffices to show that there exists a sequence $u_k\in C_{0,\gamma}^{\infty}(\R^4)\subset \dot{H}_\gamma^{1,3}(\R^4)$, $\gamma=(2, 2)$, such that the right hand side of the inequality \eqref{eq:newHardy1-3-2} goes to zero while the left hand side remain bounded from below  by  a positive number.   

Block radial functions from $\dot{H}_\gamma^{1,3}(\R^4)$ are unambiguously defined by functions on $\{(r_1,r_2): r_1,r_2\ge 0\}$. Let $(r,\theta)$, $0\le \theta< \frac \pi 2$, $r>0$, be the polar coordinates in this quadrant. 
Let  $\psi_k\in C^\infty_0(\R)$ be a sequence of positive functions of the radial variable $r$ supported near some $r_0> 1$, and let $\varphi_{k}$ be a sequence of positive, smooth compactly supported  functions of the angular variable  $\theta$. We put $u_k(r,\theta)= \varphi_k(\theta)\psi_k(r)$. Then 
\begin{align*}
\| u_k \|_{1,3}^3 = c\int\hspace{-2mm}\int
\Big( \psi'_k(r)^2 \varphi_k(\theta)^2 + \frac{1}{r^2} \varphi'_k(\theta)^2 \psi_k(r)^2 \Big)^{\frac 3 2} r^3 \sin (2\theta) dr d\theta 
\end{align*}
To estimate the above norm it is sufficient to estimate the expressions
\begin{align*}
\int|\psi'_k(r)|^3 r^3 dr \int \varphi_k(\theta)^3 \sin(2\theta) d\theta\quad \text{and}\quad 
\int \psi_k(r)^3 dr  \int|\varphi'_k(\theta)|^3 \sin (2\theta) d\theta  .
\end{align*}

Now we choose the suitable functions $\varphi_k$ and $\psi_k$. Let $\alpha\in C^\infty_0(\R)$ be a smooth function such that $0\le\alpha(t)\le 1$, $\supp \alpha =[0,1]$ and $\alpha(1/2)= 1$. We take  $\varphi_k(\theta)= \alpha(k\theta)\theta^{-\frac 1 6}$. Similarly we take $\phi\in C^\infty_0(\R)$ such that $0\le\psi(t)\le 1$, $\supp \psi =[-1,1]$ and $\psi(0)= 1$. We take  $\psi_k(r)= \psi(k^2(r-r_o))$, $r_o>1$.
Then 
\begin{align}\label{counter1}
\sup_x r_{\gamma}(x)^{N-q}|u_k(x)|^q = & \,c\, \sup_{(r,\theta)} r (\sin 2\theta)^{\frac 1 2 } \psi_k(r)^3 \alpha(k\theta)^3 \theta^{-\frac{1}{2}} \\
\sim & \, c\, \sup_{(r,\theta)} r  \psi_k(r)^3 \alpha(k\theta)^3=cr_o\psi(0)\alpha(1/2) =cr_o\,  >\,  0. \nonumber
\end{align}
On the other hand 
\begin{align}\label{counter2}
&\int(\psi'_k(r))^3 r^3 dr\int \varphi_k(\theta)^3 \sin(2\theta) d\theta \sim   \\ 
&\qquad \int(\psi'(k^2(r-r_o)))^3 (k^2r)^3 dr 
\int \alpha(k\theta)^3 \theta^{\frac{1}{2}} d\theta \le C k^{-\frac{7}{2}},\nonumber   
\intertext{and}
&\int \psi_k(r)^3 dr \int(\varphi'_k(\theta))^3 \sin(2\theta) d\theta \le \label{counter3}\\
&\qquad C k^{-2} \Big(\int k^3 |\alpha'(k\theta)|^3 \theta^{\frac{1}{2}} d\theta + \int \alpha(k\theta)\theta^{-\frac{1}{6}}d\theta\Big) \le C k^{-\frac{1}{2}} \nonumber
\end{align} 
Now \eqref{counter1}-\eqref{counter3} prove the proposition. \hfill $\Box$

\section{Convergence properties of $m$-radial functions}

We have the following corollary of Theorem \ref{thm:CKN}, showing
that vanishing of a sequence of $m$-radial functions in $L^{q^{*}}$
implies vanishing in a weighted $L^{\sigma}$ for an interval of $\sigma$
that extends above $q^{*}$.
\begin{thm}\label{leb}
Let $1<m<N$, $1\le q<N$, and let $\gamma_{i}\neq1$
whenever $q=1$. If $(u_{k})$ is a bounded sequence in $\dot{H}_{0,\gamma}^{1,q}(\R^{|\gamma|})$
and $u_{k}\to0$ in $L^{q^{*}}(\R^{|\gamma|})$, then 
\begin{equation}
\int_{\R^{|\gamma|}}\,\dfrac{|u_{k}(x)|^{\sigma}}{r_\gamma^{|\gamma|\frac{q-\sigma}{q}+\sigma}}dx\to0\label{eq:p}
\end{equation}
for any $ $ $\sigma\in(q,\frac{qm}{m-q})$ if $q<m$ or for any $\sigma>q$
otherwise.
\end{thm}
(We recall that $r_\gamma$ is defined in \eqref{weight}.) 
\begin{proof} Let $I_\sigma(u)$ be the following expression, respectively:
\begin{enumerate} 
 \item[(i)] the left hand side of (\ref{eq:newHardy1-3}), when $q\le\sigma<q^{*}$;
 \item[(ii)] the left hand side of (\ref{eq:CKN}) with $p=\frac{qm}{m-q}$, when $q<m$ and $q^{*}<\sigma<\frac{qm}{m-q}$;
 \item[(iii)] the left hand side of (\ref{eq:CKN}) with any $p>\sigma$, when $m\ge q$ and $\sigma>q^{*}$.
\end{enumerate} 
Note that in the first case $I_\sigma$ is bounded by Theorem~\ref{thm:2}, and in the second and third case by Theorem~\ref{thm:CKN}.
In all three cases the integral in (\ref{eq:p}) is an interpolation
by H\"older inequality between $I_\sigma$ and 
 $\|u_{k}\|_{q^{*}}$, which converges to zero, which proves the theorem. \qed 
\end{proof}

We now formulate a preliminary result on defect of convergence in the space 
$\dot{H}_{\gamma}^{1,q}(\R^{|\gamma|})$, based on restriction of
Solimini's profile decomposition in $\dot{H}^{1,q}(\R^{|\gamma|})$, cf. 
\cite{Solimini},  to the $m$-radial subspace.
\begin{thm}
\label{thm:pd} Assume that $1<m<N$, $1<q<N$ and that $\gamma_i\ge 2$ for every $i=1,\dots,m$. Let $(u_{k})\subset\dot{H}_{\gamma}^{1,q}(\R^{|\gamma|})$
be a bounded sequence. Then there exists a renamed subsequence of
$(u_{k})$, sequences $(j_{k}^{(n)})_{k\in\N}\subset\Z$ and functions
$w^{(n)}\in\dot{H}_{\gamma}^{1,q}(\R^{|\gamma|})$, $n\in\N$, such
that 
\begin{align}
& |j_{k}^{(n)}-j_{k}^{(m)}|\to\infty\; \text{ whenever }\; m\neq n, \nonumber \\
& 2^{-\frac{|\gamma|-q}{q}j_{k}^{(n)}}u_{k}(2^{-j_{k}^{(n)}}\cdot)\rightharpoonup w^{(n)},\nonumber\\
& u_{k}(x)=\sum_{n\in\N}2^{\frac{|\gamma|-q}{q}j_{k}^{(n)}}w^{(n)}(2^{j_{k}^{(n)}}x)+r_{k},\label{eq:pd}
\end{align}
where $r_{k}\to0$ in $L^{p}(\R^{|\gamma|},r_\gamma^{-|\gamma|\frac{q-p}{q}+p})$
for each $p>q$ if $q\ge m$ and for each $p\in(q,\frac{qm}{m-q})$
otherwise.
Moreover, the series in \eqref{eq:pd} converges in $\dot{H}_{\gamma}^{1,q}(\R^{|\gamma|})$ unconditionally, uniformly in $k$ and

\begin{equation} 
\sum_{n\in\N}\|\nabla w^{(n)}\|_q^q\le \|\nabla u_{k}\|_q^q+o(1).
\label{eq:norms}
\end{equation}
\end{thm}

\begin{remark}
It is important to stress that this is not a sharp result in a sense that one may expect that, once {\em all} concentrations are subtracted, the remainder $\omega_{k}$ converges to zero in the endpoint norm with a suitable weight, ($L^{\frac{qm}{m-q}}$ if $m>q$ or a weighted $L^\infty$ if $m<q$). Our conjecture is that "all concentrations" should for this purpose include concentrations on all singular orbits (i.e. orbits of dimension less than $|\gamma|-m$) of the symmetry group $O(\gamma_{1})\times\dots\times O(\gamma_{m})$. For example, assuming $m>q$, the sequence $u_{k}=k^{\frac{m-q}{q}}w(k(r_{1}-1),\dots,k(r_{m}-1))$
with arbitrary $w\neq0$, which consists of the remainder
alone (i.e. $w^{(n)}=0$ for all $n$), does not vanish in the weighted $L^{\frac{qm}{m-q}}$, but concentrates on a singular orbit that is not the origin. 

On the other hand, when $m=1$, the origin is the only singular orbit, and the remainder in \eqref{eq:pd} vanishes in the endpoint norm:
\begin{equation}
\sup_{r>0}r^{\frac{|\gamma|-q}{q}}|\omega_{k}(r)|\to0.\label{eq:dw}
\end{equation}
Indeed, since $\delta(r-1)$ is a continuous functional on $\dot{H}^{1,q}(\R^{|\gamma|})$
when $|\gamma|>q$, and $2^{\frac{|\gamma|-q}{q}j_{k}}\omega_{k}(2^{j_{k}}\cdot)\rightharpoonup0$
for any sequence $(j_{k})$ (as it follows from the argument of Theorem
\ref{thm:pd} repeated for $m=1$, which implies $t_{k}^{\frac{|\gamma|-q}{q}}\omega_{k}(t_{k}\cdot)\rightharpoonup0$
for any sequence $(t_{k})$ of positive numbers, $\langle\delta(\cdot-1),t_{k}^{\frac{|\gamma|-q}{q}}\omega_{k}(t_{k}\cdot)\rangle\to0$
is the same as $t_{k}^{\frac{|\gamma|-q}{q}}\omega_{k}(t_{k})\to0$,
which gives (\ref{eq:dw}). 
\end{remark}
Before we prove Theorem\ref{thm:pd}, we quote a particular case of Lemma~2.1 from \cite{skrti1}, when the manifold $M$ is $\R^{|\gamma|}$, and the group $\Omega$ is $O(\gamma_1)\times\dots\times O(\gamma_m)$. Note that this group is connected, and is coercive in the sense of Definition~1.2 of \cite{skrti1} whenever $\gamma_i\ge 2$ for $=1,\dots,m$, since then it contains $-I$, and therefore the diameter of the orbit of any given point $x$ is at least $2|x|$, which is a coercive function of $x$.
\begin{lem}
\label{separate-omega}
Assume that $\gamma_i\ge 2$ for $=1,\dots,m$. Then for any sequence $(y_k)\subset\R^{|\gamma|}$ such that $|y_k|\to\infty$, there exists  a sequence of elements $\omega_k^{(1)},\dots,$ $\omega_k^{(k)}\in O(\gamma_1)\times\dots\times O(\gamma_m)$ such that a renumbered subsequence of $y_k$ satisfies
$$|\omega^{(m)}_ky_k-\omega^{(n)}_k y_k|\to\infty \qquad \text{whenever} \qquad m\neq n.$$
\end{lem}

We can now prove Theorem\ref{thm:pd}. The proof follows the reduction approach used in \cite{Fieseler-Tintarev}, Proposition~5.1
\begin{proof} The starting point of the proof is the profile decomposition of Solimini \cite{Solimini}, amended by two elementary observations. First, without loss of generality one can replace sequences of general positive numbers $t_k^{(n)}$ with dyadic sequences $2^{j_k^{(n)}}$, $j_k^{(n)}\in\Z$, and, second,  since the remainder in \cite{Solimini} vanishes in $L^{q^*}$, it vanishes in $L^{p}(\R^{|\gamma|},r_\gamma^{-|\gamma|\frac{q-p}{q}+p})$ by Theorem~\ref{leb}. 
After this reduction, Solimini's profile decomposition takes, for a renamed subsequence, the following form
\begin{equation}
u_{k}(x)=\sum_{n\in\N}2^{\frac{|\gamma|-q}{q}j_{k}^{(n)}}w^{(n)}
(2^{j_{k}^{(n)}}x-y_k^{(n)})+r_{k},\label{eq:pd1}
\end{equation}
with $r_k\to 0$ in $L^{q^*}$, and sequences $j_{k}^{(n)}\in \Z$ and $y_{k}^{(n)}\in \R^{(|\gamma|)}$ satisfying the decoupling condition:
\begin{equation}
\label{decouple}
|y_{k}^{(m)}-y_{k}^{(n)}|+|j_{k}^{(m)}-j_{k}^{(n)}|\to\infty
\end{equation}
whenever $m\neq n$.
We will now use the $m$-radial symmetry to prove further restrictions on the terms that may appear in \eqref{eq:pd1}.
Note that if $z_k\to 0$, $w(\cdot-z_k)-w \to 0$ in the Sobolev norm, and of course this remains true if we replace $w$ with $2^{\frac{|\gamma|-q}{q}j_{k}}\bar w
(2^{j_{k}}\cdot-y_k)$ with any $j_k\in\Z$ and $y_k\in \R^{|\gamma|}$.  
Therefore, any of the terms $2^{\frac{|\gamma|-q}{q}j_{k}^{(n)}}w^{(n)}
(2^{j_{k}^{(n)}}x-y_k^{(n)})$ in  \eqref{eq:pd1} with $2^{-j_k^{(n)}}y_k^{(n)}\to 0$ may be replaced with 
$2^{\frac{|\gamma|-q}{q}j_{k}^{(n)}}w^{(n)}
(2^{j_{k}^{(n)}}x)$. To conclude the proof it remains now to show that for no $n\in\N$ there is a renamed subsequence with $2^{-j_k^{(n)}}|y_k^{(n)}|\ge \epsilon$ for some $ \epsilon>0$.
Let us fix such $n$ and a corresponding subsequence, and consider  \eqref{eq:pd1} for  $2^{-\frac{|\gamma|-q}{q}j_{k}^{(n)}}u_k(2^{-j_k^{(n)}}\cdot)$ instead of $u_k$, which allows us, without loss of generality, to assume that 
$j_k^{(1)}=0$ and $|y_k^{(1)}|\ge \epsilon$. 
Assume first that $y_k^{(1)}$ has a bounded subsequence, and, therefore, a renamed subsequence that converges to some point $y\neq 0$. 
Then $u_k(\cdot+y)\rightharpoonup w^{(1)}$, and, by the symmetry of $u_k$, this means that  $u_k(\cdot+\omega y)\rightharpoonup w^{(1)}$ for any $\omega\in SO(\gamma_1)\times\dots\times SO(\gamma_m)$. Since the weak limit is unique, we have $w^{(1)}(x)=w^{(1)}(x-\omega y)$ for each $\omega$ and for all $x\in\R^{|\gamma|}$. Noting that  the set $\{\omega_1y-\omega_2y,\quad\omega_1,\omega_2\in  SO(\gamma_1)\times\dots\times SO(\gamma_m)\}$ contains a neighbourhood of the origin, we conclude that $w^{(1)}$ is a constant, and thus, $w^{(1)}=0$. 

It remains to consider therefore the case   $|y_k^{(1)}|\to \infty$. Since for any $\omega\in SO(\gamma_1)\times\dots\times SO(\gamma_m)$, $u_k(x)=u_k(\omega^{-1} x)$, using this with  $\omega=\omega_k^{(j)}$ provided by  Lemma~\ref{separate-omega} when $y_k=y_k^{(1)}$,
we arrive at $u_k(\cdot+\omega_k^{(j)}y_k^{(1)})\rightharpoonup w^{(1)}$, $j\in\N$,  
with  $|\omega^{(m)}_ky_k^{(1)}-\omega^{(n)}_k y_k^{(1)}|\to\infty$ whenever $m\neq n$, which implies that for any $M>0$, $\|\nabla u_k\|_q^q\ge M\|\nabla w^{(1)}\|_q^q+o(1)$. Therefore, necessarily, $w^{(1)}=0$, which concludes the proof.\qed
\end{proof}

The following corollary is a slight generalization of the compactness result of Ding \cite{Ding}.
\begin{corollary}
Let $1<m<N$, $1<q<N$ and $\gamma_i\ge 2$ for $i=1,\dots,m$, and assume that $p>q$ if $q\ge m$ and $p\in(q,\frac{qm}{m-q})$ otherwise. 
Then  for any $\sigma\in(q,p)$ the embedding ${H}_{\gamma}^{1,q}(\R^{N})\hookrightarrow L^{\sigma}(\R^{N},r_\gamma^{-(\sigma-q)(\frac{p}{p-q}-\frac{|\gamma|}{q}})$
is compact. In particular, embedding into $L^s(\R^N)$ is compact for all $s\in(q,q^*)$ (as shown in \cite{Ding}).
\end{corollary}
\begin{proof}
Let $u_k\rightharpoonup 0$ in ${H}_{\gamma}^{1,q}(\R^{|\gamma|})$ and apply Theorem~\ref{thm:pd}. Note that the additional $L^q$- bound on the sequence implies that $w^{(n)}=0$ unless $j_k^{(n)}\to +\infty$. Note also that all the concentration terms with  $j_k^{(n)}\to +\infty$ in  \eqref{eq:pd} vanish in $L^{\sigma}(\R^{|\gamma|},r_\gamma^{-(\sigma-q)(\frac{p}{p-q}-\frac{|\gamma|}{q}})$. 
Thus the renamed subsequence of $u_k$ is identified with the remainder term, which by Theorem~\ref{leb} vanishes in $L^p()$. The assertion of the theorem follows by interpolating using the H\"older inequality between $L^p(\R^{|\gamma|},r_\gamma^{-|\gamma|\frac{q-p}{q}+p})$ and $L^q(\R^{|\gamma|})$. \qed
\end{proof}

\section{Appendix}
Let us present an alternative proof for a particular case \eqref{eq:newHardy111} of the Hardy-type inequality \eqref{eq:newHardy1}. The reason for giving a second proof is that it provides additional information about the inequality \eqref{eq:newHardy111}, namely that it is not sharp, but there exists a
continuous positive biradial function $W(r_{1,}r_{2})$ on $\R^{4}\setminus\{r_{1}r_{2}=0\}$
such that $Q(u)\ge\int_{\R^{4}}W(r_{1},r_{2})u(r_{1},r_{2})^{2}dx$.
This follows from the Allegretto-Piepenbrink argument, since the proof is based on construction of a positive supersolution for the associated equation, which happens not to be a solution.  See \cite{PinTin,TakTin} for details and other possible forms of the remainder.

\begin{lem}
\label{lem:newHardy0}Let $m=2$, $\gamma_{1}=\gamma_{2}=2$, and $\alpha\in\R$.
There exists a constant $C_{1}>0$, $C_{1}\ge C_{0}(\alpha-1)^{2}$, 
such that for all $u\in C_{0\gamma}^{\infty}(\R^{4}\setminus Y(\gamma))$, 
\begin{equation}
\int_{\R^{4}}\frac{|\nabla u(x)|^{2}}{(r_{1}(x)r_{2}(x))^{\alpha-1}}dx-C_{1} 
\int_{\R^{4}} \frac{|u(x){}|^{2}}{(r_{1}(x)r_{2}(x))^{\alpha}} dx\ge 0. 
\label{eq:newHardy111}
\end{equation}
\end{lem}

\begin{proof}
The function $u$ is block-radial, therefore $|\nabla_x u(x)| = |\nabla_{r_1,r_2} u(r_1,r_2)|$ and the inequality \eqref{eq:newHardy111} is equivalent to 
\begin{equation}
Q(u)=\int_0^\infty\int_0^\infty\frac{|\nabla u(r_1,r_2)|^{2}}{(r_{1}r_{2})^{\alpha-2}}dr_1 dr_2 -C_1
\int_0^\infty\int_0^\infty\frac{|u(r_1,r_2){}|^{2}}{(r_{1}r_{2})^{\alpha-1}} dr_1 dr_2 \ge 0. 
\label{eq:newHardy1a1}
\end{equation}

By the well-known argument, based on the ground state transform (\cite[Corollary 2.4]{LSV},
see also \cite{MP}, \cite{Pin} or  \cite[Theorem 8.3.4]{Davies}), it suffices to find a supersolution
to the elliptic equation corresponding to the quadratic form $Q$ in a domain $(0,\infty)^{2}\subset \R^2$.  
We consider this equation in the polar
coordinates $(r,\theta)$, $r>0$, $0<\theta<\frac{\pi}{2}$ of the quadrant $(r_{1},r_{2})\in(0,\infty)^{2}$, that is, $r_{1}=r\cos\theta$
and $r_{2}=r\sin\theta$:
\begin{equation}
-r^{2\alpha-5}\partial_{r}(r^{5-2\alpha}\partial_{r}u)-r^{-2}(\sin2\theta)^{\alpha-2}\partial_{\theta}((\sin2\theta)^{2-\alpha}\partial_{\theta}u)\ge Cr^{-2}(\sin2\theta)^{-1}u.\label{eq:supersol}
\end{equation}

A supersolution can be then given as 
\begin{equation}
u(r,\theta)=\begin{cases}
r^{-1}\sqrt{\log\left(\frac{1}{\sin2\theta}\right)}, & \alpha=1\\
r^{\alpha-2}(\sin2\theta)^{\text{\ensuremath{\frac{\alpha-1}{2}}}}, & \alpha\neq1
\end{cases}\label{eq:supersolution}
\end{equation}
and trivial but tedious calculations yield (\ref{eq:supersol}) with
$C$ equal to a constant multiple of $(\alpha-1)^{2}$ when $\alpha\neq1$. \qed
\end{proof}

\begin{remark}
The heuristic considerations that led to proposing the supersolution
(\ref{eq:supersolution}) are based on analogy with the supersolution
(which is also a solution) $u=r^{\frac{2-n}{2}}$ to the Euler-Lagrange equation associated with the radial Hardy
inequality 
\[
-\frac{1}{r^{n-2}}\partial_{r}(r^{n-1}\partial_{r}u)=\left(\frac{n-2}{2}\right)^{2}\frac{u}{r^{2}}
\]
with $5-2\alpha$ and $2-\alpha$ taking place of $n-1$, respectively
for the radial and the angular variable. In case $\alpha=1$ resp.
$n=2$ the choice of the angular part of the supersolution follows
the ground state for the Leray inequality (\cite{Leray}).
\end{remark}



\begin{thebibliography}{10}

\bibitem{AT1} L.~Abatangelo, S. Terracini, Solutions to nonlinear Schrödinger equations with singular electromagnetic potential and critical exponent, J. Fixed Point Theory Appl. \textbf{10} (2011),  147-180. 

\bibitem{AT2} L.~Abatangelo, S. Terracini, A note on the complete rotational invariance of biradial solutions to semilinear elliptic equations. Adv. Nonlinear Stud. \textbf{11} (2011), 233-245.

\bibitem{Aubin}  Th.~Aubin, Espaces de Sobolev sur les varietes riemanniennes, Bull. Sci. Math. (2) \textbf{100} (1976), 149-173.

\bibitem{CO} Y.~Cho, T.~Ozawa,  Sobolev inequalities with symmetry, Commun. Contemp. Math. \textbf{11} (2009), 355-365. 

\bibitem{CKN}L.~A.~Caffarelli, R. Kohn, L.~Nirenberg, First order
interpolation inequalities with weights, Composito Math. \textbf{53}
(1984), 259-275.

\bibitem{Davies}E.~Davis, Spectral theory and differential operators, Cambridge University Press, 1995.

\bibitem{fig}D. G. de Figueiredo, E. Moreira, O. H. Miyagaki,
Sobolev spaces of symmetric functions and applications,  
J. Func. Anal. \textbf{261} (2011), 3735-3770.

\bibitem{DeNapoli} P.~L.~De Napoli, I.~Drelichman, R.~G.~Duran,
Improved Caffarelli-Kohn-Nirenberg and trace inequalities for radial functions. 
Commun. Pure Appl. Anal. \textbf{11} (2012),  1629-1642. 


\bibitem{Ding}Y. Ding, Non-radial solutions of semilinear elliptic
equations, Math. Acta Sci \textbf{10} (1990), 229-239.

\bibitem{Ku}
I.~Kuzin, Existence theorems for nonlinear  elliptic problems in $R^N$ in the class of block-radial functions, 
Differential Equations {\bf 30} (1994), 637-646.

\bibitem{PK}
I.~Kuzin and S.~Pohozaev, Entire solutions of semilinear
elliptic equations. Birkh\"auser, Basel 1997.


\bibitem{LSV}D. Lenz, P. Stollmann, I. Veselic, The Allegretto-Piepenbrink
Theorem for Strongly Local Dirichlet Forms, Documenta Mathematica
\textbf{14} (2009) 167-189.


\bibitem{Leo}  G.~Leoni,  A first Course in Sobolev spaces. Graduate Studies in Mathematics, 105. American Mathematicas Society, Providence, Rhode Island, 2009.

\bibitem{Leray} J.~Leray, Sur le mouvement d'un liquide visqueux emplissant l'espace, Acta Math. {\bf 63} (1934), 193-248.

\bibitem{Lio82}
P.-L.~Lions, Sym\'etrie et compacit\'e dans les
espaces de Sobolev, J. Funct. Anal. {\bf 49} (1982), 315-334.

\bibitem{Mazya-book} V.~Maz'ya, Sobolev spaces with applications to elliptic partial differential equations. Second, revised and augmented edition.  Springer, Heidelberg, 2011.  

\bibitem{MP}  W.F.~Moss, J.~Piepenbrink,  Positive solutions of elliptic equations. Pacific J. Math. \textbf{75} (1978), 219-226. 


\bibitem{Pin} Y.~Pinchover,  Topics in the theory of positive solutions of second-order elliptic and parabolic partial differential equations. 
In: Spectral theory and mathematical physics: a Festschrift in honor of Barry Simon's 60th birthday, 329-355, Proc. Sympos. Pure Math., 76, Part 1, Amer. Math. Soc., Providence, RI, 2007.

\bibitem{PinTin}Y.Pinchover, K. Tintarev,  A ground state alternative for singular Schrödinger operators, J. Func. Anal. {\bf 230} (2006), 65-77.

\bibitem{SS} W.~Sickel, L.~Skrzypczak,  On the Interplay of Regularity and Decay in Case of Radial Functions II. Homogeneous spaces, J.Fourier Anal. App. \textbf{18} (2012), 548-582. 

\bibitem{SSV} W.~Sickel, L.~Skrzypczak, J.Vybiral  On the Interplay of Regularity and Decay in Case of Radial Functions I. Inhomogeneous spaces, Comm. Contemp. Math. \textbf{14}, (2012) 1250005 (60 pages). 
 
\bibitem{Solimini} S.~Solimini, A note on compactness-type properties with respect to Lorentz norms of bounded
subsets of a Sobolev space, Ann. Inst. H. Poincar\'e, Section C, \textbf{12} (1995), 319-337.

\bibitem{LS} L.~Skrzypczak, Rotation invariant subspaces of Besov and
Triebel-Lizorkin space: compactness of embeddings, 
smoothness and decay of functions, Rev. Mat. Iberoamericana, \textbf{18} (2002), 267-299.

\bibitem{skrti1} L.~Skrzypczak, C. Tintarev, A geometric criterion for compactness of invariant subspaces, Arch. Math. \textbf{101} (2013), 259-268.

\bibitem{Strauss}W. Strauss, Existence of solitary waves in higher dimensions, Comm. Math. Phys. \textbf{55} (1977), 149-162. 

\bibitem{TakTin} P.~Takac, K.~Tintarev, Generalized minimizer solutions for equations with the p-Laplacian and a potential term, Proc. Royal Soc. Edinburgh Sect. A \textbf{138} (2008), 201-221. 

\bibitem{Talenti}  G.~Talenti, Best constant in Sobolev inequality. Ann. Mat. Pura Appl. \textbf{110} (1976), 353-372. 


\bibitem{Fieseler-Tintarev}K.~Tintarev and K.-H.~Fieseler, 
Concentration compactness: functional-analytic grounds and applications,
Imperial College Press, 2007.

\end{thebibliography}
\end{document}